\documentclass{amsproc}%
\usepackage{amsfonts}
\usepackage{amsmath}
\usepackage{amssymb}
\usepackage{graphicx}
\usepackage{hyperref}%
\providecommand{\U}[1]{\protect\rule{.1in}{.1in}}
\theoremstyle{plain}

\newtheorem{definition}{Definition}
\newtheorem{example}{Example}

\newtheorem{lemma}{Lemma}

\newtheorem{proposition}{Proposition}
\newtheorem{remark}{Remark}

\newtheorem{theorem}{Theorem}
\numberwithin{equation}{section}
\begin{document}
\title[{\normalsize On $\phi$-1-absorbing prime ideals}]{{\normalsize On $\phi$-1-absorbing prime ideals}}
\author{Eda Y\i ld\i z}
\address{Department of Mathematics, Yildiz Technical University, Istanbul, Turkey.}
\email{edyildiz@yildiz.edu.tr}
\author{\"{U}nsal Tekir}
\address{Department of Mathematics, Marmara University, Istanbul, Turkey.}
\email{utekir@marmara.edu.tr}
\author{Suat Ko\c{c}}
\address{Department of Mathematics, Marmara University, Istanbul, Turkey.}
\email{suat.koc@marmara.edu.tr}

\begin{abstract}
In this paper, we introduce $\phi$-1-absorbing prime ideals in commutative
rings. Let $R$ be a commutative ring with a nonzero identity $1\neq0$ and
$\phi:\mathcal{I}(R)\rightarrow\mathcal{I}(R)\cup\{\emptyset\}$ be a function
where $\mathcal{I}(R)$ is the set of all ideals of $R$. A proper ideal $I$ of
$R$ is called a $\phi$-1-absorbing prime ideal if for each nonunits $x,y,z\in
R$ with $xyz\in I-\phi(I)$, then either $xy\in I$ or $z\in I$. In addition to
give many properties and characterizations of $\phi$-1-absorbing prime ideals,
we also determine rings in which every proper ideal is $\phi$-1-absorbing prime.

\end{abstract}
\maketitle

\section{Introduction}

Throughout the paper, we focus only on commutative rings with a nonzero
identity. Let $R$ will always denote such a ring. We will denote the set of
all ideals of $R$ by $\mathcal{I}(R)$. A proper ideal $I$ of $R$ is an element
$I\in\mathcal{I}(R)$ with $I\neq R$. For many years, numerous types of ideals
have been developed such as prime, primary, maximal, etc. All of them play
significant role when characterizing a ring. The concept of prime ideals and
its generalizations have a significant place in commutative algebra since they
are used in understanding the structure of rings. Recall that a proper ideal
$I\ $of $R\ $is said to be a \textit{prime ideal }if whenever $xy\in I$ for
some $x,y\in R,$ then either $x\in P$ or $y\in P$ \cite{AtMac}. The importance
of prime ideals led many researchers to work prime ideals and its
generalizations. See, for example, \cite{Ba2}, \cite{Be} and \cite{KoUlTe}. In
\cite{AnSmi}, Anderson and Smith introduced a notion of weakly prime ideal
which is a generalization of prime ideals. A proper ideal $I$ of $R$ is called
\textit{weakly prime ideal} if $0\neq xy\in I$ for some elements $x,y\in I$
implies that $x\in I$ or $y\in I$. They gave many results concerning weakly
prime ideals and used it to study factorization in commutative rings with zero
divisors. Also, they gave necessary and sufficient conditions so that any
proper ideal of $R$ can be written as a product of weakly prime ideals. It is
clear that every prime ideal is weakly prime but the converse is not true in
general. Afterwards, Badawi, in his celebrated paper \cite{Ba1}, introduced
the notion of 2-absorbing ideals and used them to characterize Dedekind
domains. Recall from \cite{Ba1}, that a nonzero proper ideal $I$ of $R$ is
called \textit{2-absorbing ideal} if $xyz\in I$ for some $x,y,z\in R$ implies either $xy\in I$ or $xz\in I$ or $yz\in I$. Note that every prime ideal is
also a 2-absorbing ideal. After this, over the past decades, $2$-absorbing
version of ideals and many generalizations of 2-absorbing ideals attracted
considerable attention by many researchers in \cite{BaTeYe}, \cite{TeKoOrSh}
and \cite{KoUrTe}. In \cite{Dar2}, in order to study unique factorization
domain, Bhatwadekar and Sharma defined almost prime ideals which is a
generalization of prime ideals. A proper ideal $I$ is called \textit{almost
prime ideal} if $xy\in I-I^{2}$ for some $x,y\in R$ implies that $x\in I$ or
$y\in I$. They investigated the relations among the prime ideals, pseudo prime
ideals and almost prime ideals of $R$. Badawi and Darani in \cite{BaDa}
defined and studied weakly 2-absorbing ideals which is a generalization of
weakly prime ideals. A proper ideal $I$ of $R$ is called a \textit{weakly
2-absorbing ideal} if for each $x,y,z\in R$ with $0\neq xyz\in I$, then either
we have $xy\in I\ $or $xz\in I$ or $yz\in I$.\ In \cite{AnMa}, Anderson and
Bataineh defined a new class of prime ideals. A proper ideal $I$ of $R$ is
called \textit{$\phi$-prime ideal} if whenever $xy\in I-\phi(I)$ for some
$x,y\in R$ then either $x\in I$ or $y\in I$ where $\phi:\mathcal{I}%
(R)\rightarrow\mathcal{I}(R)\cup\{\emptyset\}$ is a function. They showed that
a prime ideal and a $\phi$-prime ideal have some similar properties. Recently,
in \cite{YasNik}, Yassine et al. introduced 1-absorbing prime ideal. This type
of ideals is a generalization of prime ideals. A proper ideal $I$ of $R$ is
called \textit{1--absorbing prime ideal} if whenever $xyz\in I$ for some
nonunits $x,y,z\in R$ then either $xy\in I$ or $z\in I$. Note that every prime
ideal is 1-absorbing prime and every 1-absorbing prime ideal is 2-absorbing
ideal. The converses are not true. For instance, $P=6\mathbb{%
%TCIMACRO{\U{2124} }%
%BeginExpansion
\mathbb{Z}
%EndExpansion
}$ is a 2-absorbing ideal of $\mathbb{%
%TCIMACRO{\U{2124} }%
%BeginExpansion
\mathbb{Z}
%EndExpansion
}$ but not a 1-absorbing prime ideal and also $P=(\overline{0})$ is a
1-absorbing prime ideal of $\mathbb{%
%TCIMACRO{\U{2124} }%
%BeginExpansion
\mathbb{Z}
%EndExpansion
}_{4}$ which is not prime. They characterized 1-absorbing prime ideals of some
special rings such as valuation domain and principal ideal domain. They also
gave Prime Avoidance Theorem for 1-absorbing prime ideals. More currently,
Ko\c{c} et al. defined weakly-1-absorbing prime ideals which is a
generalization of 1-absorbing prime ideal \cite{w-1}. A proper ideal $I$ of
$R$ is called \textit{weakly-1-absorbing prime ideal} if $0\neq xyz\in I$ for
some nonunits $x,y,z\in R$ implies that $xy\in I$ or $z\in I$. They gave many
properties of this class of ideals and characterized rings that every proper
ideal is weakly-1-absorbing ideal. Moreover, they investigated
weakly-1-absorbing ideal in $C(X)$, which is the set of all real-valued
continuous functions of topological space $X$.

In this paper, we define $\phi$-1-absorbing prime ideals as a new class of
ideals which is generalization of 1-absorbing prime ideals. A proper ideal $I$
of $R\ $is called \textit{$\phi$-1-absorbing prime ideal} if whenever $xyz\in
I-\phi(I)$ for some nonunits $x,y,z\in R$ then $xy\in I$ or $z\in I$.\ Among
other results in this paper, we give some relations between $\phi$-1-absorbing
prime ideals and other classical ideals such as weakly prime ideals, $\phi
$-prime ideals, 1-absorbing prime ideals and weakly 1-absorbing prime ideals
(See, Proposition \ref{pfirst}). In particular, we show that every $\phi
$-prime ideal is also a $\phi$-1-absorbing prime ideal. But the converse is
not true in general (See, Example \ref{ex3}). Hovewer, we give a condition
under which any $\phi$-1-absorbing prime ideal is $\phi$-prime (See, Theorem
\ref{tp1-abs}). Also, we give some characterizations of $\phi$-1-absorbing
prime ideals in general rings, in factor ring, in localization of rings, in
cartesian product of rings (See, Theorem \ref{tmm}, Theorem \ref{tfac},
Theorem \ref{loc}, Theorem \ref{tgen}). Finally, we determine rings over which
every ideal is almost 1-absorbing prime ideal (See, Theorem \ref{tmain}).

\section{Characterization of $\phi$-1-absorbing prime ideals}

Let $R$ be a commutative ring. Define a function $\phi:\mathcal{I}%
(R)\rightarrow\mathcal{I}(R)\cup\{\emptyset\}$. This function maps an ideal of
$R$ to an ideal of $R$ or $\emptyset$.

\begin{definition}
Let $R$ be a ring and $I$ be a proper ideal of $R$. $I$ is called $\phi
$-1-absorbing prime ideal of $R$ if whenever $xyz\in I-\phi(I)$ for some
nonunits $x,y,z\in R$ then $xy\in I$ or $z\in I$.
\end{definition}

The following notations will be used for the rest of the paper.

\begin{example}
Let $R$ be a commutative ring and $\phi_{\alpha}:\mathcal{I}(R)\rightarrow
\mathcal{I}(R)\cup\{\emptyset\}$ be a function. The following gives types of
1-absorbing prime ideals corresponding to $\phi_{\alpha}$. \newline

$\phi_{\emptyset} \quad\quad\phi(I)=\emptyset\quad\quad\quad\quad
\text{1-absorbing prime ideal}$

$\phi_{0} \quad\quad\phi(I)=0 \quad\quad\quad\quad\text{weakly-1-absorbing
prime ideal}$

$\phi_{2} \quad\quad\phi(I)=I^{2} \quad\quad\quad\text{almost-1-absorbing
prime ideal}$

$\phi_{n} \quad\quad\phi(I)=I^{n} \quad\quad\quad\text{n-almost-1-absorbing
prime ideal}$

$\phi_{w} \quad\quad\phi(I)=\bigcap_{n=1}^{\infty}I^{n} \quad
\text{w-1-absorbing prime ideal}$

$\phi_{1} \quad\quad\phi(I)=I \quad\quad\quad\quad\text{any ideal}$\newline

Consider two functions $\phi,\psi:\mathcal{I}(R)\rightarrow\mathcal{I}%
(R)\cup\{\emptyset\}$. Then $\phi\leq\psi$ if $\phi(I)\subseteq\psi(I)$ for
all ideals of $R$. Moreover, note that $\phi_{\emptyset}\leq\phi_{0}\leq
\phi_{w}\leq\cdots\leq\phi_{n+1}\leq\phi_{n}\leq\cdots\leq\phi_{2}\leq\phi
_{1}$.

We will assume that $\phi(I)\subseteq I$ throughout the paper.
\end{example}

\begin{proposition}
\label{pfirst}(i)\ Let $I$ be a proper ideal of $R$ and $\phi,\psi
:\mathcal{I}(R)\rightarrow\mathcal{I}(R)\cup\{\emptyset\}$ be two functions with
$\phi\leq\psi$. If $I$ is a $\phi$-1-absorbing prime ideal, then $I$ is a
$\psi$-1-absorbing ideal.

(ii) $I\ $is a 1-absorbing prime ideal $\Rightarrow$ $I\ $is a weakly
1-absorbing prime ideal $\Rightarrow$ $I\ $is a $w$-1-absorbing prime ideal
$\Rightarrow$\ $I\ $is an $n$-almost 1-absorbing prime ideal for each $n\geq2$
$\Rightarrow$ $I$ is an almost 1-absorbing prime ideal.$\ $

(iii)\ $I\ $is an $n$-almost 1-absorbing prime ideal if and only for each
$n\geq2\ $if and only if $I\ $is a $w$-1-absorbing prime ideal.

(iv)\ Every $\phi$-prime ideal is a $\phi$-1-absorbing prime ideal.
\end{proposition}

\begin{proof}
(i):\ Assume that $I$ is a $\phi$-1-absorbing prime ideal. Let $xyz\in
I-\psi(I)$ for some nonunits $x,y,z\in R$. Then, $xyz\in I-\phi(I)$ and since
$I$ is $\phi$-1-absorbing ideal, $xy\in I$ or $z\in I$ which completes the proof.

(ii):\ Follows from the fact that $\phi_{\emptyset}\leq\phi_{0}\leq\phi
_{w}\leq\cdots\leq\phi_{n+1}\leq\phi_{n}\leq\phi_{2}$ and (i).

(iii):\ By (ii), we know that if $I\ $is a $w$-1-absorbing prime ideal, then
$I\ $is an $n$-almost 1-absorbing prime ideal for each $n\geq2$. Now, assume
that $I\ $is an $n$-almost 1-absorbing prime ideal if and only for each
$n\geq2.\ $Let $xyz\in I-\bigcap_{n=1}^{\infty}I^{n}$ for some nonunits
$x,y,z\in R.\ $Then there exists $m\geq2\ $such that $xyz\notin I^{m}.\ $Since
$I\ $is an $m$-almost 1-absorbing prime ideal of $R\ $and $xyz\in I-I^{m}%
,\ $then either we have $xy\in I$ or $z\in I.\ $

(iv):\ It is clear.
\end{proof}

\begin{example}
\label{ex1}\textbf{(weakly 1-absorbing prime ideal that is not 1-absorbing
prime ideal)} Let $p,q\ $be distinct prime numbers and consider the ring $R=%
%TCIMACRO{\U{2124} }%
%BeginExpansion
\mathbb{Z}
%EndExpansion
_{pq^{2}}.\ $Then $I=(\overline{0})\ $is a weakly 1-absorbing prime ideal of
$R.\ $Since $\overline{p}\overline{q}\overline{q}\in I\ $and$\ \overline
{p}\overline{q},\overline{q}\notin I,\ I\ $is not a 1-absorbing prime ideal of
$R.\ $
\end{example}

\begin{example}
\textbf{(w-1-absorbing prime ideal that is not weakly 1-absorbing prime ideal)
}Let $I\ $be an idempotent ideal of $R,\ $that is, $I=I^{2}.\ $Then $I\ $is a
w-1-absorbing prime ideal since $I^{n}=I$ for each $n\geq2.\ $But $I\ $may not
be a weakly 1-absorbing prime ideal of $R.\ $For instance, take $R=%
%TCIMACRO{\U{2124} }%
%BeginExpansion
\mathbb{Z}
%EndExpansion
_{2}^{4}$ and $I=%
%TCIMACRO{\U{2124} }%
%BeginExpansion
\mathbb{Z}
%EndExpansion
_{2}\times(0)\times(0)\times(0).\ $Then $I\ $is a w-1-absorbing prime ideal
since it is idempotent. Now, take the nonunits $x=(1,1,1,0),y=(1,1,0,1)\ $and
$y=(1,0,1,1)$ in $R.\ $Then $0\neq xyz\in I\ $but $xy,z\notin I.\ $So it
follows that $I\ $is not a weakly 1-absorbing prime ideal of $R.$
\end{example}

\begin{example}
\label{ex3}($\phi$\textbf{-1-absorbing prime ideal that is not }$\phi
$\textbf{-prime}) Take $R\ $as in Example \ref{ex1} and consider the ideal
$I=(\overline{q^{2}})$ of $R.\ $Suppose that $\phi(I)=(\overline{0}).\ $Then
$I\ $is not $\phi$-prime since $\overline{q}\overline{q}\in I-\phi(I)\ $and
$\overline{q}\notin I.\ $Now, take nonunits $\overline{x},\overline
{y},\overline{z}\in R\ $such that $\overline{0}\neq\overline{x}\overline
{y}\overline{z}\in I.\ $Then it is clear that $q^{2}|xyz$ and $pq^{2}\nmid
xyz.\ $If $q^{2}|xy,\ $then we are done. So assume that $q^{2}\nmid xy.\ $On
the other hand, since $q^{2}|xyz,\ $we have $q|z.\ $If $q^{2}|z,\ $again we
are done. So we may assume that $q|z\ $but $q^{2}\nmid z.\ $Since
$q^{2}|xyz,\ q|z\ $and $q^{2}\nmid z,\ $we have either $q|x\ $or
$q|y.\ $Without loss of generality, suppose that $q|x\ $but $q\nmid y.\ $Since
$\overline{y}\ $is not unit, we have $p|y\ $and in this case $\overline
{x}\overline{y}\overline{z}=\overline{0}\ $which is a contradiction.
Therefore, we have either $q^{2}|xy\ $or $q^{2}|z$, namely, $xy\in I\ $or
$z\in I.\ $
\end{example}

\begin{theorem}
\label{tmm}Let $R$ be a commutative ring and $I$ a proper ideal of $R$. The
following statements are equivalent.

(i) $I$ is a $\phi$-1-absorbing prime ideal of $R$.

(ii)\ For each nonunits $x,y\in R$ with $xy\not \in I$ implies $(I:xy)=I\cup
(\phi(I):xy).$

(iii)\ For each nonunits $x,y\in R$ with $xy\notin I$ gives either $(I:xy)=I$ or
$(I:xy)=(\phi(I):xy).$

(iv) For each nonunits $x,y\in R$ and proper ideal $J$ of $R$ such that
$xyJ\subseteq I$ but $xyJ\not \subseteq \phi(I)$ implies either $xy\in I$ or
$J\subseteq I$.

(v) For each nonunit $x\in R$ and proper ideals $J,K$ of $R$ such that
$xJK\subseteq I$ but $xJK\not \subseteq \phi(I)$, either $xJ\subseteq I$ or
$K\subseteq I.$

(vi) For each proper ideals $J,K,L$ of $R$ such that $JKL\subseteq I$ but
$JKL\not \subseteq \phi(I)$, either $JK\subseteq I\ $or $L\subseteq I$.
\end{theorem}

\begin{proof}
$(i)\Rightarrow(ii):$ Assume that $I$ is a $\phi$-1-absorbing ideal of $R$ and
$xy\notin I$ for some nonunit elements $x,y\in R$. It is clear that
$I\cup(\phi(I):xy)\subseteq(I:xy)$. On the other hand, choose $z\in(I:xy)$ and
so $xyz\in I$. If $xyz\not \in \phi(I)$, then $z\in I$. Now suppose
$xyz\in\phi(I)$. Then, $z\in(\phi(I):xy)$. Therefore, it gives
$(I:xy)\subseteq I\cup(\phi(I):xy)$.

$(ii)\Rightarrow(iii):$ Since $(I:xy)=I\cup(\phi(I):xy)$, $(I:xy)$ must be one
of the component in the union.

$(iii)\Rightarrow(iv):$ Assume that $xyJ\subseteq I$ but $xyJ\not \subseteq
\phi(I)$. Let $xy\not \in I$. Then, either $(I:xy)=(\phi(I):xy)$ or $(I:xy)=I$
by $(iii)$. Suppose the former case holds. Since $xyJ\subseteq I$, we have
$J\subseteq(I:xy)=(\phi(I):xy)$. It gives $xyJ\subseteq\phi(I)$ which is a
contradiction. Now, suppose the latter case holds. Then, $J\subseteq(I:xy)=I$
showing $J\subseteq I$, as needed.

$(iv)\Rightarrow(v):$ Let $xJK\subseteq I$ and $xJK\not \subseteq \phi(I)$.
Suppose $xJ\not \subseteq I$ and $K\not \subseteq I$. Then there exists $a\in
J$ such that $xa\not \in I$. Also, since $xJK\not \subseteq \phi(I)$ there
exists $b\in J$ such that $xbK\not \subseteq \phi(I)$. Now assume that
$xaK\not \in \phi(I)$. Since $x,a$ are nonunits and $xaK\subseteq I$, we have
either $xa\in I$ or $K\subseteq I$, a contradiction. So, we get $xaK\in
\phi(I)$. Also, we have $x(a+b)K\subseteq I-\phi(I)$ and it implies $x(a+b)\in
I$. Since $xbK\subseteq I-\phi(I)$ and $K\not \subseteq I$, we get $xb\in I$.
Thus, we obtain $xa\in I$ giving a contradiction. This proves $xJ\subseteq I$
or $K\subseteq I$.

$(v)\Rightarrow(vi):$ Let $JKL\subseteq I$ but $JKL\not \subseteq \phi(I)$ for
some proper ideals $J,K$ and $L$ of $R$. Assume that $JK\not \subseteq I$ and
$L\not \subseteq I$. Then, there exists $y\in J$ such that $yK\not \subseteq
I$. Also since $JKL\nsubseteq\phi(I)$, $xKL\not \in \phi(I)$ for some $x\in
J$. Then, we get $xK\subseteq I$ since $xKL\subseteq I-\phi(I)$. Suppose
$yKL\not \subseteq \phi(I)$. By $(v),$ this gives $yK\subseteq I$ or
$L\subseteq I$, which is contradiction. So, $yKL\subseteq\phi(I)$. As
$(x+y)KL\subseteq I-\phi(I)$, we have $(x+y)K\subseteq I$. This implies
$yK\subseteq I$, a contradiction.

$(vi)\Rightarrow(i):$ Let $xyz\in I-\phi(I)$. Then, $(x)(y)(z)\subseteq I$ and
$(x)(y)(z)\not \subseteq \phi(I)$. Hence, $(x)(y)\subseteq I$ or $(z)\subseteq
I$ showing that $xy\in I$ or $z\in I$, as desired.
\end{proof}

\begin{definition}
Let $I\ $be a $\phi$-1-absorbing prime ideal and $x,y,z\ $be nonunit elements
of $R.\ $If $xyz\in\phi(I),\ xy\notin I$ and $z\notin I,\ $then we say that
$(x,y,z)$ is a\ $\phi$-1-triple zero of $I.\ $
\end{definition}

\begin{remark}
(i)\ Let $I\ $be a $\phi$-1-absorbing prime ideal of $R.\ $Then $I\ $has a
$\ \phi$-1-triple zero if and only if there exists $z\notin I\ $and a nonunit
element $y\in R$ such that $(\phi(I):yz)\nsubseteq(I:y).$

(ii)\ Let $I\ $be a proper ideal of $R.\ $Then $I\ $is a 1-absorbing prime
ideal if and only if the following two conditions must be hold:

\qquad(a) $I\ $is a $\phi$-1-absorbing prime ideal of $R$.

\qquad(b) For each $z\notin I$ and nonunit element $y\in R,\ $we have
$(\phi(I):yz)\subseteq(I:y).$
\end{remark}

\begin{theorem}
\label{I3} Suppose that $I$ is a $\phi$-1-absorbing prime ideal of $R$ that is
not 1-absorbing prime and $(x,y,z)\ $is a $\phi$-1-triple zero of $I$. Then,

(i)\ $xyI\subseteq\phi(I).$\ 

(ii) If $xz,yz\notin I,\ $then $xzI,\ yzI,\ xI^{2},\ yI^{2},\ zI^{2}%
\subseteq\phi(I).\ $In particular, $I^{3}\subseteq\phi(I).\ $
\end{theorem}

\begin{proof}
(i):\ Let $I$ be a $\phi$-1-absorbing prime ideal of $R$ that is not
1-absorbing prime and $(x,y,z)\ $be a $\phi$-1-triple zero of $I$. Then we
have $xyz\in\phi(I),\ xy\notin I\ $and $z\notin I.\ $Suppose
$xyI\not \subseteq \phi(I)$. Then, there exists $a\in I$ such that
$xya\not \in \phi(I)$. So, $xy(z+a)\not \in I-\phi(I)$. If $z+a$ is unit, then
$xy\in I$,\ a contradiction. Now assume that $z+a$ is nonunit and so we get
$xy\in I$ or $z\in I,\ $again a contradiction. Thus, we have $xyI\subseteq
\phi(I)$.

(ii):\ Now, assume that $xz,yz\notin I.\ $We will show that
$xzI,\ yzI\subseteq\phi(I).\ $Suppose that $xzI\nsubseteq\phi(I).\ $Then there
exists an element $a\in I\ $such that $xza\notin\phi(I).\ $This implies that
$x(y+a)z\in I-\phi(I).\ $If $y+a$ is unit, then $xz\in I\ $which is a
contradiction. Thus $y+a\ $is nonunit. Since $I\ $is a $\phi$-1-absorbing
prime ideal,\ we conclude either $x(y+a)\in I\ $or $z\in I,\ $which implies
$xy\in I\ $or $z\in I,\ $again a contradiction. Thus, $xzI\subseteq\phi
(I).\ $By using similar argument, we have $yzI\subseteq\phi(I)$. Now, we will
show that $xI^{2}\subseteq\phi(I)$. Suppose to the contrary. Then, there
exists $a,b\in I$ such that $xab\not \in \phi(I)$. It implies $x(y+a)(z+b)\in
I-\phi(I)$. If $(y+a)$ is unit, $x(z+b)\in I$ which gives $xz\in I,\ $a
contradiction. Similarly,$\ (z+b)$ is nonunit.\ Then, either $x(y+a)\in I$ or
$z+b\in I$ implying that $xy\in I$ or $z\in I$. Thus, we have $xI^{2}%
\subseteq\phi(I)$. Similarly, we get $yI^{2}\subseteq\phi(I)$ and
$zI^{2}\subseteq\phi(I)$, we are done. For the rest, if $I^{3}\not \subseteq
\phi(I)$, there exists $a,b,c\in I$ such that $abc\not \in \phi(I)$. Then,
$(x+a)(y+b)(z+c)\in I-\phi(I)$. If $x+a\ $is unit, then we obtain
$(y+b)(z+c)=yz+yc+zb+bc\in I\ $and so $yz\in I,\ $which is a contradiction.
Similarly, we can show that $y+b\ $and $z+c\ $are nonunits.\ Then, we get
$(x+a)(y+b)\in I$ or $z+c\in I$. This gives $xy\in I$ or $z\in I$, again a
contradiction. Hence, $I^{3}\subseteq\phi(I).$
\end{proof}

\begin{theorem}
Let $R$ be a ring and $a$ be a nonunit element of $R$. Suppose that
$(0:a)\subseteq(a)$\ (e.g., $a$ is regular). Then, $(a)$ is $\phi$-1-absorbing
prime ideal with $\phi\leq\phi_{2}$ if and only if $(a)$ is a 1-absorbing
prime ideal.
\end{theorem}

\begin{proof}
If $(a)$ is 1-absorbing prime ideal, then $(a)$ is $\phi$-1-absorbing ideal.
For the other direction, assume that $(a)$ is $\phi$-1-absorbing prime ideal
with $\phi\leq\phi_{2}$. Then, it is also $\phi_{2}$-1-absorbing prime ideal
by Proposition \ref{pfirst}. Let $xyz\in(a)\ $for some nonunits $x,y,z\in R$.
If $xyz\not \in (a)^{2}$, then $xy\in(a)$ or $z\in(a)$. So suppose
$xyz\in(a)^{2}$. We have $xy(z+a)\in(a)$. If $z+a$ is unit, we are done.
Hence, we can assume $z+a$ is nonunit. Assume that $xy(z+a)\not \in (a)^{2}$.
Then we get either $xy\in(a)$ or $z+a\in(a)\ $implying $xy\in(a)$ or $z\in
(a)$. Now assume $xy(z+a)\in(a)^{2}$. This gives $ayz\in(a)^{2}$ and so there
exists $t\in R\ $such that $ayz=a^{2}t$. Thus we have $yz-at\in(0:a)\subseteq
(a).$ Therefore, $yz\in(a)+(0:a)\subseteq(a)$, as needed.
\end{proof}

Now, we give a condition for a $\phi$-1-absorbing prime ideal of $R\ $to
become a $\phi$-prime ideal of $R$.

\begin{theorem}
\label{tp1-abs}Let $I\ $be a proper ideal of a non-quasi local ring\ $R.\ $%
Suppose that $(\phi(I):x)$ is not maximal ideal for each $x\in I.\ $The
following statements are equivalent.

(i)\ $I\ $is a $\phi$-prime ideal of $R.$

(ii)\ $I\ $is a $\phi$-1-absorbing prime ideal of $R.\ $
\end{theorem}

\begin{proof}
$(i)\Rightarrow(ii):\ $Follows from Proposition \ref{pfirst}.

$(ii)\Rightarrow(i):\ $Let $I\ $be a $\phi$-1-absorbing prime ideal of
$R.\ $Choose $x,y\in R\ $such that $xy\in I-\phi(I).\ $If $x$ or $y$ is unit,
then $x\in I\ $or $y\in I\ $which is needed. So suppose that $x,y\ $are
nonunits in $R.\ $Since $xy\notin\phi(I),\ (\phi(I):xy)$ is proper. Choose a
maximal ideal $\mathfrak{m}_{1}$ of $R\ $containing $\ (\phi(I):xy)\subseteq
\mathfrak{m}_{1}.\ $Since $R\ $is non-quasi-local, there exists a different
maximal ideal $\mathfrak{m}_{2}\ $of $R.\ $Now, take $z\in\mathfrak{m}%
_{2}-\mathfrak{m}_{1}.\ $Then $z\notin\ (\phi(I):xy),\ $and so we have
$(zx)y\in I-\phi(I).\ $Since $I\ $is a $\phi$-1-absorbing prime ideal of
$R,\ $we get either $zx\in I\ $or $y\in I.\ $If $y\in I,\ $then we are done.
So assume that $zx\in I.\ $As $z\notin\mathfrak{m}_{1},\ $then there exists an
$a\in R\ $such that $1+az\in\mathfrak{m}_{1}.\ $Note that $1+az$ is nonunit.
If $1+az\notin(\phi(I):xy),\ $then we have $(1+az)xy\in I-\phi(I)$ implying
$(1+az)x\in I\ $and so $x\in I\ $since $zx\in I.\ $So assume that
$1+az\in(\phi(I):xy),\ $that is, $xy(1+az)\in\phi(I).\ $Now, choose an element
$b\in\mathfrak{m}_{1}-(\phi(I):xy).\ $Then we have $(1+az+b)xy\in
I-\phi(I).\ $On the other hand, since $1+az+b\in\mathfrak{m}_{1},\ 1+az+b$ is
nonunit. This implies that $(1+az+b)x\in I.\ $Also, since $bxy\in
I-\phi(I),\ $we get $bx\in I.\ $Then we have $x=(1+az+b)x-a(zx)-bx\in
I.\ $Therefore, $I\ $is a $\phi$-prime ideal of $R.$
\end{proof}

Now, for any ideal $J$ of $R$ define a function $\phi_{J}:\mathcal{I}%
(R/J)\rightarrow\mathcal{I}(R/J)\cup\{\emptyset\}$ by $\phi_{J}(I/J)=(\phi
(I)+J))/J$ where $J\subseteq I$ and $\phi_{J}(I/J)=\emptyset$ if
$\phi(I)=\emptyset$. Also, note that $\phi_{J}(I/J)\subseteq I/J$.

\begin{theorem}
\label{tfac}(i)\ Let $I$ be a $\phi$-1-absorbing prime ideal of $R.\ $Then
$I/\phi(I)\ $is a weakly 1-absorbing prime ideal of $R/\phi(I).\ $

(ii) Let $I/\phi(I)\ $be a weakly 1-absorbing prime ideal of $R/\phi(I)$\ and
$u(R/\phi(I))=\{x+\phi(I):x\in u(R)\}.$ Then $I\ $is a $\phi$-1-absorbing
prime ideal of $R.$

(iii)\ Let $I,J$ be two ideals of $R$ with $J\subseteq I$ and $I$ be a $\phi
$-1-absorbing prime ideal. Then, $I/J$ is a $\phi_{J}$-1-absorbing prime ideal
of $R/J$.
\end{theorem}

\begin{proof}
(i):\ Let $\overline{0}\neq\Bar{x}\Bar{y}\Bar{z}\in I/\phi(I)$for some
nonunits $\Bar{x},\Bar{y},\Bar{z}\in R/\phi(I),\ $where $\overline{x}%
=x+\phi(I),\overline{y}=y+\phi(I)$ and $\overline{z}=z+\phi(I)$. Then
$x,y,z\ $are nonunits in $R\ $and $xyz\in I-\phi(I).$ Since $I$ is a $\phi
$-1-absorbing prime ideal of $R$, $xy\in I$ or $z\in I$. Then, we get
$\overline{x}\overline{y}\in I/J$ or $\Bar{z}\in I/J$ which completes the proof.

(ii): Let $I/\phi(I)\ $be a weakly 1-absorbing prime ideal of $R/\phi(I)$\ and
$u(R/\phi(I))=\{x+\phi(I):x\in u(R)\}.\ $Choose nonunits $x,y,z$ in $R\ $such
that $xyz\in I-\phi(I).\ $Then we have $\overline{0}\neq\Bar{x}\Bar{y}\Bar
{z}\in I/\phi(I).\ $Since $u(R/\phi(I))=\{x+\phi(I):x\in u(R)\},\ \Bar{x}%
,\Bar{y}\ $and $\Bar{z}\ $are nonunits in $R/\phi(I).\ $Since $I/\phi(I)\ $is
a weakly 1-absorbing prime ideal, we have either $\Bar{x}\Bar{y}\in I/\phi(I)$
or $\overline{z}\in I/\phi(I)$, which implies $xy\in I\ $or $z\in
I.\ $Therefore, $I\ $is a $\phi$-1-absorbing prime ideal of $R.$

(iii):\ It is similar to (i)
\end{proof}

Let $R$ be a commutative ring and $S$ be a multiplicatively closed subset of
$R$. Consider the function $\phi:\mathcal{I}(R)\rightarrow\mathcal{I}%
(R)\cup\{\emptyset\}$. Define $\phi_{S}:\mathcal{I}(S^{-1}R)\rightarrow
\mathcal{I}(S^{-1}R)\cup\{\emptyset\}$ by $\phi_{S}(I)=S^{-1}\phi(I\cap R)$
and $\phi_{S}(I)=\emptyset$ if $\phi(I\cap R)=\emptyset$. Here, it is easy to
see that $\phi_{S}(I)\subseteq I$.

\begin{theorem}
\label{loc}Let $R$ be a commutative ring, $\phi:\mathcal{I}(R)\rightarrow
\mathcal{I}(R)\cup\{\emptyset\}$ be a function, $I$ be a $\phi$-1-absorbing
prime ideal of $R$ and $S$ be a multiplicatively closed subset of $R$ with
$I\cap S=\emptyset$ and $S^{-1}\phi(I)\subseteq\phi_{S}(S^{-1}I)$. Then,
$S^{-1}I$ is a $\phi_{S}$-1-absorbing prime ideal of $S^{-1}R$. Furthermore,
if $S^{-1}I\neq S^{-1}\phi(I),$ then $S^{-1}I\cap R=I$.
\end{theorem}

\begin{proof}
Let $\frac{x}{s}\frac{y}{t}\frac{z}{u}\in S^{-1}I-\phi_{S}(S^{-1}I)$ for some
nonunits $\frac{x}{s},\frac{y}{t},\frac{z}{u}\in S^{-1}R$. Then, there exists
$s^{\prime}\in S$ such that $s^{\prime}xyz\in I$ but $s^{\ast}xyz\not \in
\phi(S^{-1}I\cap R)$ for all $s^{\ast}\in S$. If $s^{\prime}xyz\in\phi(I),$
then we have $\frac{x}{s}\frac{y}{t}\frac{z}{t}\in\phi(I)_{S}\subseteq\phi
_{S}(S^{-1}I)$, a contradiction. So we get $s^{\prime}xyz=(s^{\prime}x)yz\in
I-\phi(I)$. Since $s^{\prime}x,y,z$ are nonunits in $R$ and $I$ is a $\phi
$-1-absorbing prime ideal, we get $s^{\prime}xy\in I$ or $z\in I$. This
implies $\frac{x}{s}\frac{y}{t}=\frac{s^{\prime}xy}{s^{\prime}st}\in S^{-1}I$
or $\frac{z}{u}\in S^{-1}I$.

Now we will show that $S^{-1}I\cap R=I$. Let $a\in S^{-1}I$. Then, there
exists $s\in S$ such that $sa\in I$. If $s$ is unit, we are done. If $a$ is
unit, it contrdicts with $I\cap S=\emptyset$. So we can assume $s$ and $a$ are
nonunits in $R$. If $s^{2}a=ssa\not \in \phi(I)$, we get $s^{2}\in I$ or $a\in
I$. Since former case is not possible, we have $a\in I$. In the case
$s^{2}a\in\phi(I)$, we have $a\in S^{-1}\phi(I)\cap R$. So we obtain
$S^{-1}I\cap R\subseteq I\cup\left(  S^{-1}\phi(I)\cap R\right)  $. Thus, we
conclude that either $S^{-1}I\cap R=I$ or $S^{-1}I\cap R=S^{-1}\phi(I)\cap R$.
Since latter case contradicts with the assumption, we have $S^{-1}I\cap R=I$.
\end{proof}

Let $R_{1},R_{2}$ be commutative rings and $\phi_{1}:\mathcal{I}%
(R_{1})\rightarrow\mathcal{I}(R_{1})\cup\{\emptyset\}$, $\phi_{2}%
:\mathcal{I}(R_{2})\rightarrow\mathcal{I}(R_{2})\cup\{\emptyset\}$ be two
functions. Suppose that $R=R_{1}\times R_{2}\ $and $\phi:\mathcal{I}%
(R)\rightarrow\mathcal{I}(R)\cup\{\emptyset\}$ is a function defined by
$\phi(I_{1}\times I_{2})=\phi_{1}(I_{1})\times\phi_{2}(I_{2})\ $for each ideal
$I_{k}\ $of $R_{k}.\ $Then $\phi$ is denoted by $\phi=\phi_{1}\times\phi_{2}.$

\begin{theorem}
\label{tm1}Let $R_{1},R_{2}$ be commutative rings and $\phi_{1}:\mathcal{I}%
(R_{1})\rightarrow\mathcal{I}(R_{1})\cup\{\emptyset\}$, $\phi_{2}%
:\mathcal{I}(R_{2})\rightarrow\mathcal{I}(R_{2})\cup\{\emptyset\}$ be two
functions. Suppose that $I=I_{1}\times I_{2},\ $where $I_{i}\ $is an ideal of
$R_{i}$ for each $i=1,2,\ $and $\phi=\phi_{1}\times\phi_{2}.\ $If
$I=I_{1}\times I_{2}\ $is a $\phi$-1-absorbing prime ideal of $R,$ then one of
the following three conditions must be hold.

(i) $\phi(I)=I.$

(ii)$\ I=I_{1}\times R_{2}$ and $I_{1}\ $is a $\phi_{1}$-prime ideal of
$R_{1}$ which must be prime if $\phi_{2}(R_{2})\ $is not unique maximal ideal
of $R_{2}$ (e.g. $R_{1},R_{2}\ $are not quasi-local)$.$

(iii)\ $I=R_{1}\times I_{2}$ and $I_{2}\ $is a $\phi_{2}$-prime ideal of
$R_{2}$ which must be prime if $\phi_{1}(R_{1})\ $is not unique maximal ideal
of $R_{1}$ (e.g. $R_{1},R_{2}\ $are not quasi-local)$.$
\end{theorem}

\begin{proof}
Suppose that $I$ is a $\phi$-1-absorbing prime ideal of $R.\ $First, we will
show that $I_{1}\ $is a $\phi_{1}$-prime ideal of $R_{1}.\ $To see this,
choose $x,y\in R\ $such that $xy\in I_{1}-\phi_{1}(I_{1}).\ $Then we have
$(x,0)(1,0)(y,0)=(xy,0)\in I-\phi(I)\ $for some nonunits $(x,0),(1,0),(y,0)\in
R.\ $Since $I$ is a $\phi$-1-absorbing prime ideal of $R,\ $we get either
$(x,0)(1,0)=(x,0)\in I\ $or $(y,0)\in I$ implying that $x\in I_{1}\ $or $y\in
I_{1}.\ $Therefore, $I_{1}\ $is a $\phi_{1}$-prime ideal of $R_{1}.\ $Similar
argument shows that $I_{1}\ $is a $\phi_{2}$-prime ideal of $R_{2}.\ $Now
assume that $\phi(I)\neq I.\ $Then either $\phi_{1}(I_{1})\neq I_{1}\ $or
$\phi_{2}(I_{2})\neq I_{2}.\ $Suppose that $\phi_{1}(I_{1})\neq I_{1}.\ $Then
there exists $x\in I_{1}-\phi_{1}(I_{1}).\ $This implies that
$(1,0)(1,0)(x,1)=(x,0)\in I-\phi(I).\ $Then we have either $1\in I_{1}\ $or
$1\in I_{2},\ $that is, $I_{1}=R_{1}\ $or $I_{2}=R_{2}.\ $Without loss of
generality, we may assume that $I_{1}=R_{1}.\ $Now, we will show that
$I=R_{1}\times I_{2}\ $and $I_{2}$ is prime in $R_{2}\ $if $\phi_{1}(R_{1}%
)\ $is not unique maximal ideal of $R_{1}.\ $Let $ab\in I_{2}\ $for some
elements $a,b\in R_{2}.$\ If $a$ or $b$ is unit, we are done. So assume that
$a,b\ $are nonunits in $R_{2}.\ $Since $\phi_{1}(R_{1})\ $is not unique
maximal ideal of $R_{1}$,$\ $there exists a nonunit element $x\in R_{1}%
-\phi_{1}(R_{1}).\ $Then we have $(x,1)(1,a)(1,b)=(x,ab)\in I-\phi(I).\ $Since
$I$ is a $\phi$-1-absorbing prime ideal of $R,\ $we have either
$(x,1)(1,a)=(x,a)\in I\ $or $(1,b)\in I$ implying or $a\in I_{2}\ $or $b\in
I_{2}.\ $Therefore, $I_{2}\ $is a prime ideal of $R_{2}.\ $
\end{proof}

Recall that a commutative ring $R\ $is said to be a \textit{quasi-local} if it
has a unique maximal ideal \cite{Sharp}. Otherwise, we say $R\ $is not
quasi-local or non-quasi-local.

\begin{theorem}
\label{tnloc}Let $R_{1},R_{2}$ be commutative rings such that $\phi_{i}%
(I_{i})$ is not unique maximal ideal of $R_{i}\ $(e.g. $R_{i}\ $is not
quasi-local) and $\phi_{i}:\mathcal{I}(R_{i})\rightarrow\mathcal{I}(R_{i}%
)\cup\{\emptyset\}$ for each $i=1,2$. Suppose that $I=I_{1}\times I_{2}$ is
nonzero ideal$,\ $where $I_{i}\ $is an ideal of $R_{i}$ for each
$i=1,2,\ \phi=\phi_{1}\times\phi_{2}$ and $\phi(I)\neq I.\ $Then the following
statements are equivalent.

(i)\ $I\ $is a $\phi$-1-absorbing prime ideal of $R=R_{1}\times R_{2}.$

(ii)\ $I=I_{1}\times R_{2}\ $for some prime ideal $I_{1}\ $of $R_{1}\ $and
$I=R_{1}\times I_{2}\ $for some prime ideal $I_{2}\ $of $R_{2}.$

(iii)\ $I\ $is a prime ideal of\ $R.\ $

(iv)\ $I\ $is a weakly prime ideal of $R.\ $

(v)\ $I\ $is a $1$-aborbing prime ideal of $R.$
\end{theorem}

\begin{proof}
$(i)\Rightarrow(ii):\ $Follows from Theorem \ref{tm1}.

$(ii)\Rightarrow(iii):\ $Clear.

$(iii)\Leftrightarrow(iv):\ $Follows from \cite[Theorem 7]{AnSmi}.

$(iii)\Rightarrow(v):\ $Follows from \cite[Definition 2.1]{YasNik}.

$(v)\Rightarrow(i):\ $Follows from the fact that $\phi_{\emptyset}\leq\phi$
and Proposition \ref{pfirst}.
\end{proof}

\begin{theorem}
\label{tgen}Let $R_{1},R_{2}$ be commutative rings such that $\phi_{i}(I_{i})$
is not unique maximal ideal of $R_{i}\ $(e.g. $R_{i}\ $is not quasi-local) and
$\phi_{i}:\mathcal{I}(R_{i})\rightarrow\mathcal{I}(R_{i})\cup\{\emptyset\}$
for each $i=1,2,\ldots,n.\ $Suppose that $I=I_{1}\times I_{2}\times
\cdots\times I_{n}$ is nonzero ideal$,\ $where $I_{i}\ $is an ideal of $R_{i}$
for each $i=1,2,\ldots,n,\ \phi=\phi_{1}\times\phi_{2}\times\cdots\times
\phi_{n}$ and $\phi(I)\neq I.\ $Then the following statements are equivalent.

$I\ $is a $\phi$-1-absorbing prime ideal of $R=R_{1}\times R_{2}.$

(ii)\ $I=R_{1}\times R_{2}\times\cdots\times R_{t-1}\times I_{t}\times
R_{t+1}\times\cdots\times R_{n}\ $for some prime ideal $I_{t}\ $of $R_{t}%
$\ and $1\leq t\leq n\ .$

(iii)\ $I\ $is a prime ideal of\ $R.\ $

(iv)\ $I\ $is a weakly prime ideal of $R.\ $

(v)\ $I\ $is a $1$-absorbing prime ideal of $R.$
\end{theorem}

\begin{proof}
We use induction on $n.$ If $n=1,\ $the claim is clear. If $n=2,\ $the claim
follows from Theorem \ref{tnloc}. Now, assume that $(i)\Leftrightarrow
(ii)\Leftrightarrow(iii)\Leftrightarrow(iv)\Leftrightarrow(v)\ $is true for
all $k<n.\ $Let $I^{\prime}=I_{1}\times I_{2}\times\cdots\times I_{n-1}%
,\ R^{\prime}=R_{1}\times R_{2}\times\cdots\times R_{n-1}$ and $\phi^{\prime
}=\phi_{1}\times\phi_{2}\times\cdots\times\phi_{n-1}.\ $Then note that
$I=I^{\prime}\times I_{n},\ R=R^{\prime}\times R_{n}\ $and $\phi=\phi^{\prime
}\times\phi_{n}.\ $The rest follows from induction hypothesis and Theorem
\ref{tnloc}.
\end{proof}

\begin{lemma}
Let $(R,m)$ be a quasi-local ring and $\mathfrak{m}^{3}\subseteq\phi(I)$ for
every proper ideal $I$ of $R$. Then, every proper ideal of $R$ is a $\phi
$-1-absorbing prime ideal.
\end{lemma}

\begin{proof}
Let $I$ be a nonzero proper ideal of $R$. Assume that $I$ is not $\phi
$-1-absorbing prime ideal. Then, there exist nonunit elements $x,y,z\in R$
such that $xyz\in I-\phi(I)$ but $xy\not \in I$ and $z\not \in I$. Since
$x,y,z$ are nonunits, they are elements of $\mathfrak{m}$. So, $xyz\in
\mathfrak{m}^{3}\subseteq\phi(I)$, a contradiction.
\end{proof}

A ring $R\ $is said to be an \textit{indecomposable ring} if its all
idempotents are $0\ $and $1.$Otherwise, we say $R\ $is decomposable. It is
well know that a ring $R\ $is decomposable if and only if $R=R_{1}\times
R_{2}\ $for some commutative rings $R_{1}\ $and $R_{2}.$

Recall that a commutative ring $R\ $is said to be a von Neumann regular ring
if its each ideal is idempotent, or equivalently, for each $x\in R,\ $there
exists an idempotent element $e\in R\ $such that $(x)=(e)\ $\cite{von}. The
concept of von Neumann regular rings and its generalizations have drawn
considerable interest and have been widely studied by many authors. See, for
example, \cite{AnChu}, \cite{JaTe} and \cite{JaTeKo}. Now, in the following,
we characterize all rings over which every proper ideal is almost 1-absorbing
prime ideal.

\begin{theorem}
\label{tmain}Let $R\ $be a ring. Then every proper ideal is almost
$1$-absorbing prime if and only $(R,\mathfrak{m}\mathcal{)}\ $is either
quasi-local with $\mathfrak{m}^{3}=(0)$ or $R\ $is a von Neumann regular ring.
\end{theorem}

\begin{proof}
$(\Leftarrow):\ $Suppose that $(R,\mathfrak{m}\mathcal{)}\ $is quasi-local
with $\mathfrak{m}^{3}=(0).\ $Then by previous Lemma, every ideal is almost
$1$-absorbing prime. If $R\ $is von Neumann regular ring, then every ideal is
idempotent so that every ideal is almost $1$-absorbing prime.

$(\Rightarrow):\ $Now, suppose that every proper ideal is almost $1$-absorbing
prime. First, we will show that $(a^{3})=(a^{4})$ for each element $a\in
R.\ $If $a$ is unit, then we are done. So assume that $a$ is not unit. Take a
maximal ideal $\mathfrak{m}$ of $R.\ $If $a\notin\mathfrak{m}\mathcal{,}$ then
$\frac{a}{1}$ is unit in $R_{\mathfrak{m}}\ $so that we have $(a^{3}%
)_{\mathfrak{m}}=(a^{4})_{\mathfrak{m}}.\ $So suppose that $a\in
\mathfrak{m}\mathcal{.}$Then by Theorem \ref{loc}, every proper ideal of
$R_{\mathfrak{m}}$ is almost 1-absorbing prime. Since $\frac{a^{3}}{1}%
\in(\frac{a^{3}}{1})$ and $(\frac{a^{3}}{1})$ is almost 1-absorbing prime, we
have either $\frac{a^{2}}{1}\in(\frac{a^{3}}{1})$ or $\frac{a^{3}}{1}\in
(\frac{a^{3}}{1})^{2},\ $which implies that $(a^{3})_{\mathfrak{m}}%
=(a^{4})_{\mathfrak{m}}.$\ Since $(a^{3})_{\mathfrak{m}}=(a^{4})_{\mathfrak{m}%
}$ for each maximal ideal $\mathfrak{m}\mathcal{\ }$of $R,$ we have
$(a^{3})=(a^{4})$ and thus $(a^{3})=(a^{3})^{2}.\ $This implies that
$(a^{3})=(e)\ $for some idempotent $e\in R.\ $If $R\ $is not decomposable
ring, then for each nonunit $a\in R,\ a^{3}=(0)\ $and this shows that
$(R,\mathfrak{m}\mathcal{)}\ $is quasi-local with $\mathfrak{m}^{3}%
=(0),\ $where $\mathfrak{m}=\sqrt{0}.\ $Now, suppose that $R=R_{1}\times
R_{2}\ $for some commutative rings $R_{1}\ $and $R_{2}.\ $If $R_{1}\ $is not
von Neumann regular, then there exists an ideal $I\ $of $R\ $such that
$I^{2}\neq I.\ $Now take the ideal $J=I\times0\ $of $R.\ $Since $J\ $is almost
1-absorbing prime, by Theorem \ref{tm1}, $J=I\times0=I\times R_{2}\ $which is
a contradiction. Thus $R_{1}\ $is a von Neumann regular ring. Similarly,
$R_{2}\ $is von Neumann regular ring and so is $R=R_{1}\times R_{2}.$
\end{proof}

\end{document}